\documentclass[a4paper,reqno,oneside]{amsart}
\usepackage[english]{babel}
\usepackage{amsmath,esint}
\usepackage{amssymb}
\usepackage{amsthm}
\usepackage{physics}
\usepackage{mathrsfs}
\usepackage{enumerate}
\usepackage{array}   
\newcolumntype{L}{>{$}l<{$}} 
\usepackage[hidelinks]{hyperref}
\usepackage[nospace,noadjust]{cite}

\newcommand{\R}{\mathbb{R}}
\newcommand{\C}{\mathbb{C}}

\newcommand{\s}[1]{\mathcal{ #1 }}
\newcommand{\bb}[1]{\mathbb{ #1 }}

\newcommand{\fr}[1]{ \mathfrak{ #1} }
\newcommand{\vph}{\varphi}
\newcommand{\vth}{\vartheta}
\newcommand{\vep}{\varepsilon}
\newcommand{\ol}[1]{\overline{ #1 }}
\newcommand{\ox}{ \otimes }

\newcommand{\pr}{\prime}

\usepackage[left=80pt,right=80pt,top=80pt,bottom=94pt,heightrounded]{geometry}

\begin{document}

\newtheorem{thm}{Theorem}[section]
\newtheorem*{thm*}{Main Theorem}

\newtheorem{prop}[thm]{Proposition}

\newtheorem{lem}[thm]{Lemma}

\newtheorem{cor}[thm]{Corollary}
\newtheorem*{cor*}{Corollary}

\theoremstyle{definition}
\newtheorem{dfn}[thm]{Definition}

\theoremstyle{remark}
\newtheorem{rmk}[thm]{Remark}
\newtheorem{ex}[thm]{Example}

\title{Strong $1$-Boundedness of Unimodular Free Orthogonal Quantum Groups}

\author{Floris Elzinga}
\address{Department of Mathematics, University of Oslo, P.O box 1053, Blindern, 0316 Oslo, Norway}
\email{florise@math.uio.no}

\begin{abstract}
Recently, Brannan and Vergnioux showed that the free orthogonal quantum group factors $\mathcal{L}\mathbb{F}O_M$ have Jung's strong $1$-boundedness property, and hence are not isomorphic to free group factors.
We prove an analogous result for the other unimodular case, where the parameter matrix is the standard symplectic matrix in $2N$ dimensions $J_{2N}$.
We compute free derivatives of the defining relations by introducing self-adjoint generators through a decomposition of the fundamental representation in terms of Pauli matrices, resulting in $1$-boundedness of these generators.
Moreover, we prove that under certain conditions, one can add elements to a $1$-bounded set without losing $1$-boundedness.
In particular this allows us to include the character of the fundamental representation, proving strong $1$-boundedness.
\end{abstract}

\maketitle

\section{Introduction}

The C$^*$-algebras and von Neumann algebras associated to discrete groups form a rich and important class of examples.
The theory of discrete quantum groups, dual to Woronowicz's compact quantum groups \cite{Wo87,Wo98}, has in recent years proven itself to be another fruitful source of interesting C$^*$-algebras and von Neumann algebras.
The discrete duals of the free orthogonal and free unitary quantum groups of Van Daele and Wang \cite{W95,VDW96}, depending on an invertible complex $N\times N$ matrix parameter $Q$, have been particularly well studied.

Write $\bb{F}O(Q)$ for the free orthogonal quantum group associated to a general $Q$ and let $J_{2N}$ be the standard symplectic matrix in $2N$ dimensions.
We will use the notations $\bb{F}O_N = \bb{F}O(I_N)$ and $\bb{F}O^J_{2N} = \bb{F}O(J_{2N})$ for the unimodular free orthogonal quantum groups.
These two cases are of particular interest, as their associated quantum group von Neumann algebras $\s{L}\bb{F}O_N$ and $\s{L}\bb{F}O^J_{2N}$ share many properties with the free group factors \cite{B96,BC07,VV07,B12,F13,B14,DCFY14,FV15,I15,C18}.
Whether or not they could be isomorphic to a free group factor $\s{L}\bb{F}_M$ remained open for over 20 years, until it was recently settled for $Q = I_N$ by Brannan and Vergnioux \cite{BV18}.
They distinguish $\s{L}\bb{F}O_N$ from the free group factors by proving that it satisfies strong $1$-boundedness, a free probabilistic property due to Jung \cite{J07}.
The main result of the present paper is that this property also holds when $Q = J_{2N}$.

\begin{thm*}[See Theorem \ref{thm:MR}]
The free orthogonal quantum group von Neumann algebras $\s{L}\bb{F}O^J_{2N}$ are strongly $1$-bounded for $N\geq 2$.
\end{thm*}
\noindent
Combined with the work of Brannan and Vergnioux, this yields the following corollary.

\begin{cor*}[See Corollary \ref{cor:CTMR}]
Let $Q\in \mathrm{GL}_N(\C)$, $N\geq 3$, be such that $Q\ol{Q} \in \C I_N$ and such that $\bb{F}O(Q)$ is unimodular.
Then $\s{L}\bb{F}O(Q)$ is not isomorphic to any finite von Neumann algebra admitting a tuple of self-adjoint generators whose (modified) free entropy dimension exceeds $1$.
In particular this excludes being isomorphic to a(n interpolated) free group factor.
\end{cor*}

Evidence pointing towards this outcome had already appeared in the literature.
Vergnioux \cite{Ve12} and Bichon \cite{B13} proved that the first $L^2$-Betti number vanishes for both $\bb{F}O_N$ and $\bb{F}O^J_{2N}$.
Using this, it can be shown that Voiculescu's modified microstates free entropy dimension $\delta_0$ and non-microstates free entropy dimension $\delta^*$ \cite{V3,V5} give different results for the canonical set of generators in $\s{L}\bb{F}O_N$ or $\s{L}\bb{F}O^J_{2N}$, and $\s{L}\bb{F}_M$ respectively \cite{BCV17}.

It is unknown whether or not free entropy dimension is a von Neumann algebra invariant in general, but this is the case for strongly $1$-bounded von Neumann algebras \cite{J07}.
In a finite von Neumann algebra $\s{M}$ with faithful normal tracial state $\tau$, a finite tuple $X_1,\dots,X_n\in\s{M}$ of self-adjoint elements is called \emph{$1$-bounded} (without the `strong') if it satisfies a condition that is slightly stronger than $\delta_0(X_1,\dots,X_n) \leq 1$ (see Section \ref{ssec:FPaDCO}).
If $\s{M}$ admits self-adjoint generators $X_1,\dots,X_n$ that form a $1$-bounded tuple, and at least one of the $X_i$ has finite free entropy, $\s{M}$ is said to be \emph{strongly $1$-bounded}.
Jung introduced these definitions and showed that for a strongly $1$-bounded von Neumann algebra $\s{N}$, any finite set of self-adjoint generators $Y_1,\dots,Y_m\in \s{N}$ must satisfy $\delta_0(Y_1,\dots,Y_m) \leq 1$.
This forbids $\s{N}$ being isomorphic to any interpolated free group factor $\s{L}\bb{F}_r$ for $1 < r \leq \infty$ \cite[Section 3]{J07}.

Checking directly that the canonical generators of $\s{L}\bb{F}O_N$ and $\s{L}\bb{F}O^J_{2N}$ form a $1$-bounded set turns out to be difficult.
Instead, the strategy of \cite{BV18} for $\bb{F}O_N$ relies on results of Jung \cite{J16} and Shlyakhtenko \cite{S16}.
The quantum group von Neumann algebra $\s{L}\bb{F}O_N$ has $N^2$ self-adjoint operators $u = (u_{ij})_{i,j=1}^N$ as its canonical set of generators.
These generators satisfy some polynomial relations $F$, i.e.\ $F(u) = 0$ in $\s{L}\bb{F}O_N$.
One then considers the free derivatives $\partial F(u)$ of the relations $F$ with respect to the generators $u_{ij}$.
The results of Jung and Shlyakhtenko now say that in order to conclude $1$-boundedness of $u$, it is sufficient to prove that the operator $D = \partial F(u)^*\partial F(u)$ is of determinant class and has rank $N^2-1$ (see Section \ref{ssec:FPaDCO} for details).

Brannan and Vergnioux achieve this by computing the operator $D$ and relating it to something called the \emph{edge-reversing operator} on the \emph{quantum Cayley tree} due to Vergnioux \cite{Ve05,Ve12}.
Regularity results for this edge-reversing operator are proved in \cite{BV18} for many $\bb{F}O(Q)$, including the cases $Q = I_N,J_{2N}$.
The computation of the rank of $D$ proceeds by expressing the rank in terms of $L^2$-Betti numbers, which are known for all free orthogonal quantum groups.
To complete the proof, there are calculations by Banica, Collins, and Zinn-Justin \cite{BCZJ9} which imply that every $u_{ij}$ individually has finite free entropy.

There are two obstacles to generalising this proof to the case of $\bb{F}O^J_{2N}$.
The first is that the canonical generators are no longer self-adjoint, complicating the determination of $\partial F$.
We will remedy this by choosing a convenient set of self-adjoint generators using a decomposition of the fundamental representation in terms of Pauli matrices, which have simple algebraic properties and relations.
Fortunately, the connection to the edge-reversing operator remains intact, allowing us to conclude that our new set of generators is $1$-bounded.

The second obstacle is that calculations like \cite{BCZJ9} are not available for $\bb{F}O^J_{2N}$.
We sidestep this by proving a technical result of independent interest, inspired by a relative free entropy estimate due to Voiculescu \cite{V3}.
This lemma states that under certain regularity conditions, one is allowed to add redundant elements to a generating set without spoiling $1$-boundedness.
This works in particular if the redundant element is a noncommutative polynomial in the generators.
It is a result of Banica that the character of the fundamental representation of $\bb{F}O^J_{2N}$ is a semicircular element \cite{B96}, and hence possesses finite free entropy.
As the fundamental character is a linear combination of generators, we have completed the proof.
Note that this method also applies to $\bb{F}O_N$, removing the dependence on the non-trivial results of \cite{BCZJ9}.

The remainder of this paper is structured as follows.
In Section \ref{sec:prelim}, we recall the necessary facts and definitions about free orthogonal quantum groups, their corepresentation theory, quantum Cayley graphs, and free probability.
In Section \ref{sec:GR1B}, we introduce generators for $\s{L}\bb{F}O^J_{2N}$, compute their free derivatives, and show how this results in $1$-boundedness.
In Section \ref{sec:ADG}, we prove a technical lemma stating conditions under which one is allowed to enlarge a $1$-bounded set without destroying $1$-boundedness.
Finally, in Section \ref{sec:MR} we prove our main result and discuss some consequences.

\textbf{Acknowledgements:} The author wishes to thank his supervisor Makoto Yamashita for many valuable discussions and suggesting the topic.

\section{Preliminaries}\label{sec:prelim}

We will keep our notations and conventions close to \cite{BV18}.
Generally, the letters $H$, $K$, and $L$ represent (separable) Hilbert spaces, and $\s{K}(H)$ or $\s{U}(H)$ denotes the compact or unitary operators on the Hilbert space $H$ respectively.
All von Neumann algebras are assumed to have a separable predual.
We write $H\ox K$ for the tensor product of Hilbert spaces, and the same symbol is also used for the minimal tensor product of C$^*$-algebras.
Put $\Sigma$ for the map $H\ox K\rightarrow K\ox H$ that flips the tensor legs.
The Greek letter $\iota$ will be used as a generic symbol for any identity map.
We will also make use of leg numbering notation, which we will explain by example.
If $x,y$ are elements of a unital algebra $\s{A}$, then $\s{A}^{\ox 3}\ni (x\ox y)_{31} = y\ox 1\ox x$, while $\s{A}^{\ox 4}\ni (x\ox y)_{13} = x\ox 1\ox y\ox 1$, and so on.
It will always be clear from the context in which space the tensors lie.
For an operator $V$ on $H\ox H$, we have for instance that $V_{32} = \iota\ox(\Sigma V\Sigma)$ on $H\ox H\ox H$.
We write $I_N$ for the $N\times N$ identity matrix and $J_{2N}$ denotes the standard $2N\times 2N$ symplectic matrix
\begin{align*}
J_{2N} = \begin{pmatrix} 0_N & I_N \\ -I_N & 0_N \end{pmatrix} .
\end{align*}

\subsection{Free Orthogonal Quantum Groups}\label{ssec:FOQG}

For brevity, we will discuss discrete quantum groups within the context of $\bb{F}O(Q)$.

\begin{dfn}
Let $N\geq 2$ and $Q\in \mathrm{GL}_N(\C)$ such that $Q\ol{Q}\in\C I_N$, where the bar denotes taking the adjoint (i.e.\ complex conjugate) entry-wise.
Then the \emph{free orthogonal} quantum group $\bb{F}O(Q)$ is given by the unital Woronowicz C$^*$-algebra
\begin{align}
C^*\bb{F}O(Q) = \left\langle u_{ij} \bigm| 1\leq i,j\leq N, ~ u ~ \mathrm{unitary} , ~ Q \ol{u} Q^{-1} = u  \right\rangle , \label{eq:WCSA}
\end{align}
where $u$ denotes the matrix $(u_{ij})_{ij}\in M_N(\C)\ox C^*\bb{F}O(Q)$.
The matrix $u$ is the \emph{fundamental representation} of $\bb{F}O(Q)$, and the \emph{coproduct} $\Delta\colon C^*\bb{F}O(Q)\rightarrow C^*\bb{F}O(Q)\ox C^*\bb{F}O(Q)$ takes the form
\begin{align*}
\Delta(u_{ij}) = \sum_{k=1}^N u_{ik}\ox u_{kj}
\end{align*}
on its entries.
The coproduct $\Delta$ is a co-associative unital $\ast$-homomorphism satisfying the \emph{cancellation property} that the subspaces
\begin{align*}
\mathrm{span}\left\{ (x\ox 1)\Delta(y) \mid x,y\in C^*\bb{F}O(Q) \right\} &\subset C^*\bb{F}O(Q)\ox C^*\bb{F}O(Q) , \\
\mathrm{span}\left\{ (1\ox x)\Delta(y) \mid x,y\in C^*\bb{F}O(Q) \right\} &\subset C^*\bb{F}O(Q)\ox C^*\bb{F}O(Q) ,
\end{align*}
are dense.
\end{dfn}

These algebras come with a unique invariant state $h$, called the \emph{Haar} state, where invariance means that $(h\ox\iota)\Delta(x) = h(x)1 = (\iota\ox h)\Delta(x)$ for all $x\in C^*\bb{F}O(Q)$.
If $h$ is a trace, then $\bb{F}O(Q)$ is said to be \emph{unimodular}.
It is known (see \cite[Section 9.1]{B17}) that $\bb{F}O(Q)$ is unimodular when either $Q = I_N$ or $Q = J_{2N}$ (up to isomorphism).
Hence we introduce the special notations $\bb{F}O_N = \bb{F}O(I_N)$ and $\bb{F}O^J_{2N} = \bb{F}O(J_{2N})$.

One also has an involutive $\ast$-anti-automorphism $R$ of $C^*\bb{F}O(Q)$ such that $\Delta R = (R\ox R)\Sigma\Delta$, called the \emph{unitary antipode}.
The ordinary \emph{antipode} $S$ is an anti-automorphism of the $\ast$-algebra generated by the $u_{ij}$ with the property that $(\iota\ox S)(u) = u^*$.
In the unimodular case, the maps $R$ and $S$ are the same.

Applying the GNS construction to the Haar state $h$ gives a Hilbert space $\ell^2\bb{F}O(Q) = H_Q$ with canonical cyclic unit vector $\xi_0$ implementing $h$ as a vector state.
This representation gives rise to the \emph{reduced} quantum group C$^*$-algebra $C_r^*\bb{F}O(Q)$ and the \emph{quantum group von Neumann algebra} $\s{L}\bb{F}O(Q)$ in the usual ways.

On $C_r^*\bb{F}O(Q)$, the comultiplication $\Delta$ is implemented by an operator $V\in\s{U}(H_Q\ox H_Q)$ as $\Delta(y) = V(y\ox 1)V^*$.
This \emph{multiplicative unitary} $V$ is defined explicitly by $V(x\xi_0\ox y\xi_0) = \Delta(x)(1\ox y)(\xi_0\ox\xi_0)$ for $x,y\in C^*\bb{F}O(Q)$, and witnesses the \emph{pentagon equation} $V_{12}V_{13}V_{23} = V_{23}V_{12}$.
The unitary antipode $R$ descends to give an involutive unitary $U$ on $H_Q$ by $U(x\xi_0) = R(x)\xi_0$ for $x\in C^*\bb{F}O(Q)$.

We recall some facts about the free orthogonal quantum groups and the parallels to the free group factors on the von Neumann algebraic level.
If one takes an identity matrix $I_N$ in the Definition \eqref{eq:WCSA} above, the orthogonal free quantum groups $\bb{F}O_N$ are obtained.
This family is both a liberation of $C(O_N)$ and its diagonal elements (setting all off-diagonal elements to zero) are related to the full group C$^*$-algebra of the $N$-fold free product group $\bb{Z}_2\ast\cdots\ast\bb{Z}_2$ \cite{W95}.
This explains the $\bb{F}$ and the $O$ appearing in $\bb{F}O_N$.

As we are taking the point of view of discrete quantum groups, we use the notation $C^*\bb{F}O_N$ to underline the analogy with the full group C$^*$-algebra mentioned above.
If one takes the point of view of compact quantum groups instead, the notation $C^*\bb{F}O_N = C^u(O_N^+)$ is more natural in light of the relation to the orthogonal group $O_N$.
The original notation $A_o(N)$ (and more generally $A_o(Q)$) of van Daele and Wang is also common.
For general $Q$, we have a family of deformations of this Woronowicz C$^*$-algebra that still satisfy many of the same properties.

The analogy with free groups becomes stronger when one considers approximation properties.
It is a result of Banica \cite{B96} that $\bb{F}O(Q)$ is `generically' non-amenable, that is if and only if $N\geq 3$.
De Commer, Freslon, and Yamashita \cite{DCFY14} proved that $\bb{F}O(Q)$ has the Haagerup property and is weakly amenable with Cowling--Haagerup constant $1$ (also referred to as the CCAP or CMAP), generalising results by Brannan \cite{B12} and Freslon \cite{F13}.

This trend continues on the von Neumann algebraic level.
By \cite{I15,C18,FV15} it holds that $\s{L}\bb{F}O(Q)$ is strongly solid and has no Cartan subalgebra.
With some restrictions on $Q$, Vaes and Vergnioux \cite{VV07} showed that $\s{L}\bb{F}O(Q)$ is a full factor and hence prime.
In particular, if $QQ^* = I_N$ and $N\geq 3$, then $\s{L}\bb{F}O(Q)$ is a factor of type II$_1$.
Recall that $\bb{F}O(Q)$ is unimodular for $Q = I_N,J_{2N}$.
Thus the analogy between the free orthogonal quantum group von Neumann algebras $\s{L}\bb{F}O_N$ and $\s{L}\bb{F}O_{2N}^J$ on one hand and the free group factors $\s{L}\bb{F}_M$ on the other is especially striking.
It was even shown that the series $\{\s{L}\bb{F}O_N\}$ has free group factor-like asymptotics in a strong sense \cite{BC07,B14}.

\subsection{Corepresentations}\label{ssec:CoRep}

All constructions in this section are general, but we state them for $\bb{F}O(Q)$.
We refer to \cite{NT13} for the general theory of the representation categories of discrete and compact quantum groups.

A \emph{unitary corepresentation} of $\bb{F}O(Q)$ on a Hilbert space $H$ is defined as a unitary operator $v$ which lies in the multiplier algebra $M(\s{K}(H)\ox C^*\bb{F}O(Q))$ and which interacts with the comultiplication as $(\iota\ox\Delta)v = v_{12}v_{13} \in M(\s{K}(H)\ox C^*\bb{F}O(Q)\ox C^*\bb{F}O(Q))$.
The fundamental representation $u$ and the multiplicative unitary $V$ are important examples.

Taking all finite dimensional unitary corepresentations of $\bb{F}O(Q)$ as objects and their intertwiners as morphisms yields a rigid C$^*$-tensor category when equipped with the obvious direct sum and the tensor product $v\ox w = v_{13}w_{23}$.
Write $v_\mathrm{triv}$ for the trivial corepresentation on $\C$ represented by $1\in C^*\bb{F}O(Q)$, and choose a set of representatives Irr$(Q)$ of the irreducible corepresentations such that $u$ and $v_\mathrm{triv}$ are among them.
If $v\in\mathrm{Irr}(Q)$, write $H_v$ for its Hilbert space.

The algebraic direct sum $\bigoplus_{v\in\mathrm{Irr}(Q)} B(H_v)$ is dense in $H_Q$.
Restricting the multiplicative unitary $V$ to this subspace gives the decomposition $V = \sum_{v\in\mathrm{Irr}(Q)} v$ acting by left multiplication.
Using the $c_0$ direct sum instead, one forms the \emph{dual algebra} $c_0(\bb{F}O(Q)) = c_0(Q) = \bigoplus_{v\in\mathrm{Irr}(Q)}^{c_0} B(H_v)$, again acting by left multiplication on the subspace defined above.
It turns out that $V\in M(c_0(Q)\ox C^*_r\bb{F}O(Q))$.
There are two minimal central projections $p_0,p_1\in Z(M(c_0(Q)))$ such that $p_0 H_Q = B(H_{v_\mathrm{triv}}) \cong \C\xi_0$ and $p_1 H_Q = B(H_u) \cong M_N(\C)$.
Note that $p_0 p_1 = 0$ and $Up_1 = p_1 U$.

\subsection{Quantum Cayley Trees}\label{ssec:QCT}

To the pair $\bb{F}O(Q)$ and $p_1$, one can associate a \emph{quantum Cayley tree} \cite{Ve05}.
This consists of the following four pieces of data.
We have the Hilbert spaces $H_Q$ and $K_Q = H_Q\ox p_1 H_Q$, to be thought of as the vertex and edge spaces respectively.
There is a bounded linear operator $E$ from $K_Q$ to $H_Q\ox H_Q$, called the \emph{boundary operator}, given by restricting the multiplicative unitary $V$ to $K_Q$.
Finally, we have the important \emph{edge-reversing operator} $\Theta = \Sigma(1\ox U)V(U\ox U)\Sigma \in B(K_Q)$ (this uses $Up_1 = p_1 U$).
Note that $\Theta$ need not be involutive, but it is unitary.

Let us explain how this generalises the classical Cayley graph.
Let $G$ be a discrete group, and consider its group C$^*$-algebra $C^*G$ with the coproduct $\Delta(g) = g\ox g$.
It is easy to see that $\Delta$ is \emph{cocommutative}, that is $\Sigma\Delta = \Delta$.
A standard fact in this context is that the unitary antipode $R$ is given by $R(g) = g^{-1}$.
Passing to the reduced group C$^*$-algebra $C^*_r G$, we write $\{\delta_g\}$ for the orthonormal basis of $\ell^2 G$ given by the point-indicator sequences, and $\lambda\colon C^* G\rightarrow B(\ell^2 G)$ for the left regular representation.
The definition of the multiplicative unitary $V$ becomes
\begin{align*}
V(\lambda_g\delta_e\ox\lambda_h\delta_e) = (\lambda_g\ox\lambda_g)(1\ox\lambda_h)(\delta_e\ox\delta_e) = (\lambda_{g}\delta_e\ox\lambda_{gh}\delta_e) .
\end{align*}

The vertex Hilbert space is now just $\ell^2 G$.
The right analogue of $p_1$ in this context turns out to be the indicator sequence of a set $H\subset G$, not containing the neutral element $e$ and closed under inverses.
As the boundary operator $E$ is just a restriction of $V$, we see that the `boundary' of an edge $(\delta_g\ox\delta_h)$ is $(\delta_g\ox\delta_{gh})$.
Thus we should view $(\delta_g\ox\delta_h)$ as an edge in the classical Cayley graph that starts at $g$, and whose endpoint is given by right translating by $h$, i.e.\ $gh$.
Accordingly, the edge-reversing operator acts as
\begin{align*}
\Theta(\delta_g\ox\delta_h) = \Sigma(1\ox U)V(\delta_{h^{-1}}\ox\delta_{g^{-1}}) = \Sigma(1\ox U)(\delta_{h^{-1}}\ox \delta_{h^{-1}g^{-1}}) = \delta_{gh}\ox\delta_{h^{-1}}.
\end{align*}

\subsection{Free Probability and Determinant Class Operators}\label{ssec:FPaDCO}

Throughout this section $(\s{M},\tau)$ is a finite von Neumann algebra with faithful normal tracial state $\tau$.
Let $X_1,\dots,X_n$ and $Y_1,\dots,Y_m$ be self-adjoint elements in $\s{M}$.
In \cite{V2}, Voiculescu introduced the \emph{microstates free entropy} $\chi(X_1,\dots,X_n)$.
This relies on the notion of microstates $\Gamma(X_1,\dots,X_n ; \ell, k, \vep)$ of $X_1,\dots,X_n$, which are $n$-tuples of $k\times k$ self-adjoint complex matrices that approximate the moments of the $X_i$ up to degree $\ell$ within precision $\vep$.
The microstates free entropy $\chi$ is then a normalised limit over the logarithm of the volume of sets of microstates.

For later use, we state a finiteness result for the microstates free entropy of a single self-adjoint element $X\in\s{M}$.
It is a direct consequence of the formula
\begin{align*}
\chi(X) = \iint \log\abs{s-t} \mathrm{d}\mu_X(s) \mathrm{d}\mu_X(t) + \frac{3}{4} + 2^{-1} \log(2\pi) ,
\end{align*}
which can be found in Proposition 4.5 of \cite{V2}.

\begin{lem}\label{lem:fomfe}
Let $X = X^* \in \s{M}$ and write $\mu_X$ for its spectral distribution with respect to $\tau$.
If $\mu_X$ admits an essentially bounded density with respect to the Lebesgue measure on $\R$, then $\chi(X)$ is finite.
\end{lem}

We next recall the \emph{relative} microstates free entropy $\chi(X_1,\dots,X_n : Y_1,\dots,Y_m)$ \cite{V3}.
This is defined in the same way, except one considers \emph{relative} microstates $\Gamma(X_1,\dots,X_n : Y_1,\dots,Y_m ; \ell, k, \vep)$.
These are the projections onto the first $n$ factors of the microstates $\Gamma(X_1,\dots,X_n,Y_1,\dots,Y_m ; \ell, k, \vep)$.
We record some of its properties that will be used later.

\begin{prop}\label{prop:RMFEP}
The relative microstates free entropy satisfies
\begin{enumerate}[(i)]
\item Domination by the microstates free entropy and global upper bound
\begin{align*}
\chi(X_1,\dots,X_n : Y_1,\dots, Y_n) \leq \chi(X_1,\dots,X_n) \leq \frac{n}{2} \log\left[ \frac{2\pi e}{n} \tau\left( X_1^2 + \dots + X_n^2 \right) \right] .
\end{align*}
\item $\chi$ is `subadditive'
\begin{align*}
\chi(X_1,\dots,X_n : Y_1,\dots,Y_m) \leq \chi(X_1,&\dots,X_p : X_{p+1},\dots,X_n,Y_1,\dots,Y_m) \\
&+ \chi(X_{p+1},\dots,X_n : X_1,\dots,X_p,Y_1,\dots,Y_m) .
\end{align*}
\item Let $Z_1,\dots,Z_q\in\s{M}$ be self-adjoint and lying in the von Neumann algebra generated by $Y_1,\dots,Y_m$, then
\begin{align*}
\chi(X_1,\dots,X_n : Y_1,\dots,Y_m) \leq \chi(X_1,\dots,X_n : Z_1,\dots,Z_q) .
\end{align*}
\item If $Y_p,\dots,Y_m$ lie in the von Neumann algebra generated by $X_1,\dots,X_n,Y_1,\dots,Y_{p-1}$, we have
\begin{align*}
\chi(X_1,\dots,X_n : Y_1,\dots,Y_m) = \chi(X_1,\dots,X_n : Y_1,\dots,Y_{p-1}) .
\end{align*}
\end{enumerate}
\end{prop}

This leads us to the definition of the \emph{modified free entropy dimension} $\delta_0$ \cite{V3}.
Without loss of generality (replacing $\s{M}$ be a free product if necessary) we can assume that there is a free family of standard semicircular elements $S_1,\dots,S_n$ that are also free from the $X_i$.
Now define
\begin{align}
\delta_0(X_1,\dots,X_n) = n + \limsup_{\vep\rightarrow 0}\frac{\chi(X_1 + \vep S_1,\dots,X_n + \vep S_n\colon S_1,\dots,S_n)}{\abs{\log\vep}} . \label{eq:DOMFED}
\end{align}
It turns out that $\delta_0(X_1,\dots,X_n) \leq n$, and this inequality is saturated when the $X_i$ form a free standard semicircular family.
Thus the free group factor $\s{L}\bb{F}_M$ admits an $M$-tuple of generators such that their modified free entropy dimension is precisely $M$.

An important goal of free probability theory is to decide whether $\delta_0$ is a von Neumann algebraic invariant.
That is, is it true that when $X_1,\dots,X_n$ and $Y_1,\dots,Y_m$ generate isomorphic von Neumann algebras, then $\delta_0(X_1,\dots,X_n) = \delta_0(Y_1,\dots,Y_m)$?
An affirmative answer to this would solve the long-standing free group factor isomorphism problem.

Jung made progress in this direction when he introduced the notion of strong $1$-boundedness and showed that every generating set of a strongly $1$-bounded von Neumann algebra has modified free entropy dimension less than 1 \cite{J07}.
Hence any such von Neumann algebra is not isomorphic to a free group factor $\s{L}\bb{F}_M$ with $M\geq 2$.
The most convenient definition in our case is not the original one, but rather the equivalent final bullet point of Corollary 1.4 in \cite{J07}.

\begin{dfn}\label{dfn:aBS1B}
Let $\alpha>0$, then $X_1,\dots,X_n$ is \emph{$\alpha$-bounded} if
\begin{align}
\limsup_{\vep\rightarrow 0}\left[ \chi(X_1 + \vep S_1,\dots,X_n + \vep S_n \colon S_1,\dots,S_n) + (n-\alpha) \abs{\log\vep} \right] < \infty . \label{eq:DOAB}
\end{align}
If in addition to being $1$-bounded, at least one of the $X_i$ satisfies $\chi(X_i) > -\infty$, we say that $X_1,\dots,X_n$ are \emph{strongly $1$-bounded}.
\end{dfn}

Comparing \eqref{eq:DOAB} with the definition \eqref{eq:DOMFED} of $\delta_0$, one sees that $\alpha$-boundedness is a strengthening of the estimate $\delta_0(X_1,\dots,X_n) \leq \alpha$.
An alternate way to state the definition of $\alpha$-boundedness is to say that for small $\vep$ there is a constant $K\geq 0$, depending only on the $X_i$, such that
\begin{align*}
\chi(X_1 + \vep S_1,\dots,X_n + \vep S_n \colon S_1,\dots,S_n) \leq (\alpha-n)\abs{\log\vep} + K .
\end{align*}
Recalling Lemma \ref{lem:fomfe}, upgrading $1$-boundedness to strong $1$-boundedness can be achieved by showing that one of $X_i$ has a sufficiently regular spectral measure $\mu_{X_i}$.

\begin{rmk}
There is another approach to defining a free notion of entropy, called $\chi^*$, also due to Voiculescu \cite{V5}.
Instead of going through microstates, $\chi^*$ is defined through the notions of conjugate variables and free Fisher information.
This leads to a non-microstates free entropy dimension $\delta^*$, and an analogous definition of $\alpha$-boundedness for $\delta^*$.
It is a deep result of Biane, Capitaine, and Guionnet \cite{BCG03} that $\chi^*(X_1,\dots,X_n) \geq \chi(X_1,\dots,X_n)$ (and so also larger than the relative microstates free entropy).
Consequently, $\alpha$-boundedness for $\delta^*$ implies $\alpha$-boundedness for $\delta_0$.
\end{rmk}

In the remainder of this section, let us introduce some terminology necessary to state a result of Jung \cite{J16} reproved by Shlyakhtenko \cite{S16}.

Let $T_1,\dots,T_n$ be formal noncommuting indeterminates, and write $\C\langle T_1,\dots,T_n\rangle$ for their unital algebra of noncommutative polynomials.
For each $1\leq i\leq n$, define a map
\begin{align*}
\partial_i\colon\C\langle T_1,\dots,T_n\rangle\rightarrow\C\langle T_1,\dots,T_n\rangle\ox\C\langle T_1,\dots,T_n\rangle ,
\end{align*}
by the relations
\begin{align*}
\partial_i T_j &= \delta_{ij}(1\ox 1),&
\partial_i(P_1 P_2) = (\partial_i P_1)(1\ox P_2) + (P_1\ox 1)(\partial_i P_2) ,
\end{align*}
where $P_1,P_2\in\C\langle T_1,\dots,T_n\rangle$.
When we equip $\C\langle T_1,\dots,T_n\rangle^{\ox 2}$ with the $\C\langle T_1,\dots,T_n\rangle$-bimodule structure $P_1\cdot(P_2\ox P_3)\cdot P_4 = (P_1 P_2 \ox P_3 P_4)$, the $\partial_i$ become derivations.

For a vector of such polynomials $P = (P_1,\dots,P_m)\in\C\langle T_1,\dots,T_n\rangle^m$, we define
\begin{align*}
\partial P = \sum_{i=1}^n \sum_{j=1}^m (\partial_i P_j)\ox e_j\ox e_i^* \in \C\langle T_1,\dots,T_n\rangle^{\ox 2}\ox M_{m\times n}(\C) .
\end{align*}
We now want to evaluate such expressions in self-adjoint $X_1,\dots,X_n\in\s{M}$, where $\s{M}$ is still a finite von Neumann algebra with faithful normal tracial state $\tau$.
This results in $\partial P(X_1,\dots,X_n)$, which we view as an element in $\s{M}\ox\s{M}^\mathrm{op}\ox M_{m\times n}(\C)$.
Equip $L^2\s{M}\ox L^2\s{M}^\mathrm{op}$ with the right $\s{M}\ox\s{M}^\mathrm{op}$-module structure $(\xi\ox\eta)\cdot(x\ox y^\mathrm{op}) = (\xi x\ox y^\mathrm{op}\eta)$.
Then $\partial P(X_1,\dots,X_n)$ is a bounded right $\s{M}\ox\s{M}^\mathrm{op}$-module map from $L^2\s{M}\ox L^2\s{M}^\mathrm{op}\ox\C^n$ to $L^2\s{M}\ox L^2\s{M}^\mathrm{op}\ox\C^m$.
Consequently, we can define the \emph{rank} of $\partial P(X_1,\dots,X_n)$, denoted rank$(\partial F(X_1,\dots,X_n))$, as the Murray--von Neumann dimension of the closure of its image.

Finally, recall that when $A\in M_n(\C)$ is strictly positive we have the identity
\begin{align*}
\det(A) = \exp(\Tr(\log(A))) .
\end{align*}
This motivates the definition of the \emph{Fuglede--Kadison--L\"{u}ck determinant} $\det_{FKL}$ on $(\s{M},\tau)$.
Let $x\in\s{M}$, and write $\mu_{\abs{x}}$ for the spectral distribution of $\abs{x}$ with respect to $\tau$.
Then
\begin{align*}
\sideset{}{_{FKL}}\det(x) = \exp\left( \int_{0^+}^\infty \log(s)\mathrm{d}\mu_{\abs{x}}(s) \right) ,
\end{align*}
when the integral is finite, and zero else.
We say that $x$ is of \emph{determinant class} (with respect to $\tau$) if $\det_{FKL}(x) \neq 0$.

\begin{thm}[{\cite[Theorem 6.9]{J16}} and {\cite[Theorem 2.5]{S16}}]\label{thm:CF1BD}
Let $\s{M}$ be a finite von Neumann algebra with faithful normal tracial state $\tau$, and $X_1,\dots,X_n\in\s{M}$ self-adjoint.
Assume that there is a vector $F\in\C\langle T_1,\dots,T_n\rangle^m$ such that \begin{align*}
F(X_1,\dots,X_n) = 0 ~ \text{and} ~  \sideset{}{_{FKL}}\det\left[\partial F(X_1,\dots,X_n)^*\partial F(X_1,\dots,X_n)\right] \neq 0 .
\end{align*}
Then it holds that $X_1,\dots,X_n$ are $\alpha$-bounded (for both $\delta_0$ and $\delta^*$) with
\begin{align*}
\alpha = n - \mathrm{rank}\left(\partial F(X_1,\dots,X_n)\right) .
\end{align*}
\end{thm}

\section{Generators, Relations, and $1$-Boundedness}\label{sec:GR1B}

\subsection{Generators}\label{ssec:Gen}

We now fix $Q = J_{2N}$ and consider $\bb{F}O^J_{2N} = \bb{F}O(J_{2N})$.
Recall the $2N\times 2N$ matrix of canonical generators $u$.
Let us split $u$ up into four $N\times N$ pieces as
\begin{align*}
u = \begin{pmatrix} u_{(1)} & u_{(2)} \\ u_{(3)} & u_{(4)} \end{pmatrix} .
\end{align*}
Writing out the last relation in the definition \eqref{eq:WCSA} of $C^*\bb{F}O^J_{2N}$, one obtains
\begin{align*}
\begin{pmatrix} u_{(1)} & u_{(2)} \\ u_{(3)} & u_{(4)} \end{pmatrix} = \begin{pmatrix} \ol{u_{(4)}} & - \ol{u_{(3)}} \\ -\ol{u_{(2)}} & \ol{u_{(1)}} \end{pmatrix} .
\end{align*}
Therefore, $u$ must be of the form
\begin{align}
u = \begin{pmatrix} A^u + i C^u & B^u + i D^u \\ - B^u + i D^u & A^u - i C^u \end{pmatrix} , \label{eq:uitopm}
\end{align}
where $A^u,\dots,D^u$ are $N\times N$ matrices of self-adjoint operators (consisting of real and imaginary parts of the canonical generators) from $C^*\bb{F}O^J_{2N}$.
Thus $\ol{A^u} = A^u$, and so on, and we write $(A^u)_{ij} = a^u_{ij}$ ($1\leq i,j\leq N$), and so on.
The reasons for this slightly clunky notation will become clear in the next section.
We use the convention that the alphabetical indices $i,j,k,\cdots$ run from $1$ to $N$, and Greek indices from the beginning of the alphabet (e.g., $\alpha,\beta,\gamma,\dots$) run over $\{a,b,c,d\}$.
Motivated by the above, we will usually interpret $M_{2N}(\C) \cong M_2(\C)\ox M_N(\C)$.

The above form \eqref{eq:uitopm} for $u$ can be nicely expressed in terms of the matrices
\begin{align*}
\tau_a &= I_2 ,&
\tau_b &= i\sigma_y ,&
\tau_c &= i\sigma_z ,&
\tau_d = i\sigma_x ,
\end{align*}
where $\sigma_{x,y,z}$ are the Pauli matrices
\begin{align*}
\sigma_x &= \begin{pmatrix} 0 & 1 \\ 1 & 0 \end{pmatrix} ,&
\sigma_y &= \begin{pmatrix} 0 & -i \\ i & 0 \end{pmatrix} ,&
\sigma_z &= \begin{pmatrix} 1 & 0 \\ 0 & -1 \end{pmatrix} .
\end{align*}
Namely,
\begin{align}
u = \tau_a A^u + \tau_b B^u + \tau_c C^u + \tau_d D^u = \sum_{ij}^\alpha \left( \tau_\alpha\ox E_{ij}\ox\alpha^u_{ij} \right) = \sum_{ij}^\alpha \left( E_{ij}^\alpha \ox\alpha^u_{ij} \right) . \label{eq:nefu}
\end{align}
Here, we have suppressed the tensor products in the first equality (an abuse of notation we will keep committing), used the standard matrix units $E_{ij} \in M_N(\C)$ in the second, and defined $E_{ij}^\alpha = \tau_\alpha\ox E_{ij}$ in the last.
Thus we are using the $E_{ij}^\alpha$ as our basis for $M_{2N}(\C)$.
Notice that in this form
\begin{align}
u^* &= \tau_a (A^u)^t - \tau_b (B^u)^t - \tau_c (C^u)^t - \tau_d (D^u)^t  \nonumber \\
&= \sum_{ij} \left( E_{ij}^a \ox a^u_{ji} - E_{ij}^b \ox b^u_{ji}- E_{ij}^c \ox c^u_{ji}- E_{ij}^d \ox d^u_{ji} \right) . \label{eq:nefus}
\end{align}

\begin{rmk}
As an aside, it already follows from the proof of Theorem 5.1 in \cite{BCV17} that $\delta_0$ and $\delta^*$ of this set of generators is 1 (but $1$-boundedness is of course slightly stronger than this).
To see this, note that the inequality (13) above the aforementioned theorem collapses due to the vanishing of the $L^2$-Betti numbers of $\bb{F}O^J_{2N}$ \cite{B13}.
To obtain Connes embeddability of $\s{L}\bb{F}O^J_{2N}$, notice that it lies inside the graded twist $\s{L}\bb{F}O_{2N}\rtimes \bb{Z}_2$ (where $\bb{Z}_2$ acts on $u$ by conjugating with $J_{2N}$), which is in $\s{L}\bb{F}O_{2N}\ox M_2(\C)$ obtained by the crossed product by the dual action.
This last algebra is Connes embeddable because $\s{L}\bb{F}O_{2N}$ is.
\end{rmk}

\subsection{Relations}\label{ssec:Rel}

In this section we compute the free derivatives of the defining relations with respect to the generators fixed in the previous section.
Let $F = (F^{(1)},F^{(2)})$ be the vector containing the defining relations \eqref{eq:WCSA}, in the form $F(u) = 0$.
So $F^{(1)}(u) = u^* u - I_{2N}$ and $F^{(2)}(u) = uu^* - I_{2N}$.
Here, $F(u)$ is shorthand for $F(u_{11},\dots,u_{2N,2N})$, and similar notation will be used throughout the remainder of the paper.

Let $a_{ij},\dots,d_{ij}$, $1\leq i,j\leq N$, be $(2N)^2$ self-adjoint noncommuting formal indeterminates, and set $\s{C} = \C\langle a_{11},\dots,d_{NN} \rangle$.
When we evaluate in the actual operators, $c_{k\ell}$ will for instance correspond to $c_{k\ell}^u$.
Accordingly, collect the formal indeterminates into matrices $A = \sum_{ij} a_{ij}\ox E_{ij}$ and so on.
Thus we view $F\in\s{C}\ox\left(M_{2N}(\C)\oplus M_{2N}(\C)\right)$, where we consider $M_{2N}(\C)$ to just be a linear space.
Keeping in mind Equations \eqref{eq:nefu} and \eqref{eq:nefus}, we get the explicit polyonomials
\begin{align*}
F^{(1)} &= \left( A^t\tau_a - B^t\tau_b - C^t\tau_c - D^t\tau_d \right)\left( A\tau_a + B\tau_b + C\tau_c + D\tau_d \right) - I_{2N} , \\
F^{(2)} &= \left( A\tau_a + B\tau_b + C\tau_c + D\tau_d \right)\left( A^t\tau_a - B^t\tau_b - C^t\tau_c - D^t\tau_d \right) - I_{2N} .
\end{align*}
When evaluating, we will take the generators $a^u_{ij},\dots,d^u_{ij}$ in their `reduced' form acting on $H$.
This is due to the fact that we want to investigate properties of the von Neumann algebra $\s{L}\bb{F}O^J_{2N}$, which is represented on $H$, the GNS space of $C^*\bb{F}O^J_{2N}$ coming from the Haar state.

Our goal in this section is to determine
\begin{align*}
\partial F(A^u,B^u,C^u,D^u)\in B(H)\ox B(H)\ox B(M_{2N}(\C) ; M_{2N}(\C)\oplus M_{2N}(\C)),
\end{align*}
and express it in terms of the quantum group theoretic data coming from $\bb{F}O^J_{2N}$.
The result is stated in the lemma below, whose proof constitutes one of the main technical components of this article and should be viewed as analogous to \cite[Lemma 4.2]{BV18}.
Recall from Section \ref{ssec:CoRep} that there is a copy $M_{2N}(\C)\cong p_1 H$.
This identification will be important for the next lemma.

\begin{lem}\label{lem:DFCTERO}
On $H\ox H\ox p_1 H$ it holds that
\begin{align*}
\partial F^{(1)}(A^u,B^u,C^u,D^u)^* \partial F^{(1)}(A^u,B^u,C^u,D^u) = 2 + 2 \fr{Re}[W] ,
\end{align*}
where $W = V_{31}(1\ox U\ox U) V_{32}(1\ox U\ox 1)$.
The same relation is true for $F^{(2)}$.
\end{lem}

\begin{proof}
Since we are going to take free derivatives of $F^{(1)}$ and $F^{(2)}$, we can ignore the $I_{2N}$ terms.
Let us first focus on $F^{(2)}$, which can be written out using the algebraic relations of the $\tau$'s to read
\begin{align*}
F^{(2)} = F^{(2)}_a \tau_a - F^{(2)}_b \tau_b - F^{(2)}_c \tau_c - F^{(2)}_d \tau_d ,
\end{align*}
with
\begin{align*}
F^{(2)}_a &= AA^t + BB^t + CC^t + DD^t , &
F^{(2)}_b &= AB^t + DC^t - BA^t - CD^t , \\
F^{(2)}_c &= AC^t + BD^t - CA^t - DB^t , &
F^{(2)}_d &= AD^t + CB^t - DA^t - BC^t .
\end{align*}
Now, by definition $\partial F^{(2)}$ is the map such that
\begin{align*}
\partial F^{(2)} (E_{ij}^\alpha) = \sum_{k\ell}^\beta \partial_{ij}^\alpha \left( F^{(2)} \right)_{k\ell}^\beta .
\end{align*}
Here, $\partial^a_{ij}$ for instance refers to taking the free partial derivative with respect to $a_{ij}$.
By linearity of $\partial$, we can compute the free derivatives of the four pieces $F^{(2)}_{a,b,c,d}$ separately.

We perform the computation for $F^{(2)}_a$ in detail, the others are similar.
By definition
\begin{align*}
\partial_{ij}^\alpha \left( F^{(2)}_a\tau_a \right)_{k\ell}^\beta &= \delta_{a\beta} \partial_{ij}^\alpha \biggl[ \biggl( \sum_m^\gamma \gamma_{km}\gamma_{\ell m} \biggr) \ox E_{k\ell}^a \biggr] \\
&= \delta_{a\beta} \left( \sum_m \left[ \delta_{ik}\delta_{jm} (1\ox\alpha_{\ell m}) + \delta_{i\ell}\delta_{jm} (\alpha_{km}\ox 1) \right] \right)\ox E_{k\ell}^a \\
&= \delta_{a\beta} \left[ \delta_{ik} (1\ox\alpha_{\ell j}) + \delta_{i\ell} (\alpha_{kj}\ox 1) \right]\ox E_{k\ell}^a .
\end{align*}
So that
\begin{align*}
\left[ \partial \left( F^{(2)}_a\tau_a \right) \right]\left( E_{ij}^\alpha \right) &= \sum_\ell \left( 1\ox\alpha_{\ell j}\ox E_{i\ell}^a \right) + \sum_k \left( \alpha_{kj}\ox 1\ox E_{ki}^a \right) .
\end{align*}
Now notice that
\begin{align*}
E_{i\ell}^a &= (T\lambda_{\ell j}T\ox \vth_{a,\alpha}) E_{ij}^\alpha ,&
E_{ki}^a &= \left( \lambda_{kj}T\ox \vth_{a,\alpha} \right) E_{ij}^\alpha ,
\end{align*}
where $T$ and $\lambda_{ij}$ are the transpose map and left multiplication by $E_{ij}$ respectively, acting on $M_N(\C)$, and $\vth_{\alpha,\beta}$ is the rank one operator on $M_2(\C)$ that sends $\tau_\beta$ to $\tau_\alpha$.
Thus
\begin{align*}
\partial\left( F^{(2)}_a\tau_a \right) = \sum_{ij}^\alpha \left( 1\ox\alpha_{\ell j}\ox T\lambda_{ij}T\ox \vth_{a,\alpha} \right) + \sum_{k\ell}^\beta \left( \beta_{k\ell}\ox 1\ox \lambda_{k\ell}T\ox \vth_{a,\beta} \right)
\end{align*}
Analogously one finds that
\begin{align*}
\partial\left( F^{(2)}_b\tau_b \right) = &+ \sum_{ij} \left( 1\ox b_{ij}\ox T\lambda_{ij}T\ox \vth_{b,a} \right) - \sum_{k\ell} \left( b_{k\ell}\ox 1\ox \lambda_{k\ell}T\ox \vth_{b,a} \right) \\
&- \sum_{ij} \left( 1\ox a_{ij}\ox T\lambda_{ij}T\ox \vth_{b,b} \right) + \sum_{k\ell} \left( a_{k\ell}\ox 1\ox \lambda_{k\ell}T\ox \vth_{b,b} \right) \\
&- \sum_{ij} \left( 1\ox d_{ij}\ox T\lambda_{ij}T\ox \vth_{b,c} \right) + \sum_{k\ell} \left( d_{k\ell}\ox 1\ox \lambda_{k\ell}T\ox \vth_{b,c} \right) \\
&+ \sum_{ij} \left( 1\ox c_{ij}\ox T\lambda_{ij}T\ox \vth_{b,d} \right) - \sum_{k\ell} \left( c_{k\ell}\ox 1\ox \lambda_{k\ell}T\ox \vth_{b,d} \right) ,
\end{align*}
\begin{align*}
\partial\left( F^{(2)}_c\tau_c \right) = &+ \sum_{ij} \left( 1\ox c_{ij}\ox T\lambda_{ij}T\ox \vth_{c,a} \right) - \sum_{k\ell} \left( c_{k\ell}\ox 1\ox \lambda_{k\ell}T\ox \vth_{c,a} \right) \\
&+ \sum_{ij} \left( 1\ox d_{ij}\ox T\lambda_{ij}T\ox \vth_{c,b} \right) - \sum_{k\ell} \left( d_{k\ell}\ox 1\ox \lambda_{k\ell}T\ox \vth_{c,b} \right) \\
&- \sum_{ij} \left( 1\ox a_{ij}\ox T\lambda_{ij}T\ox \vth_{c,c} \right) + \sum_{k\ell} \left( a_{k\ell}\ox 1\ox \lambda_{k\ell}T\ox \vth_{c,c} \right) \\
&- \sum_{ij} \left( 1\ox b_{ij}\ox T\lambda_{ij}T\ox \vth_{c,d} \right) + \sum_{k\ell} \left( b_{k\ell}\ox 1\ox \lambda_{k\ell}T\ox \vth_{c,d} \right) ,
\end{align*}
\begin{align*}
\partial\left( F^{(2)}_d\tau_d \right) = &+ \sum_{ij} \left( 1\ox d_{ij}\ox T\lambda_{ij}T\ox \vth_{d,a} \right) - \sum_{k\ell} \left( d_{k\ell}\ox 1\ox \lambda_{k\ell}T\ox \vth_{d,a} \right) \\
&- \sum_{ij} \left( 1\ox c_{ij}\ox T\lambda_{ij}T\ox \vth_{d,b} \right) + \sum_{k\ell} \left( c_{k\ell}\ox 1\ox \lambda_{k\ell}T\ox \vth_{d,b} \right) \\
&+ \sum_{ij} \left( 1\ox b_{ij}\ox T\lambda_{ij}T\ox \vth_{d,c} \right) - \sum_{k\ell} \left( b_{k\ell}\ox 1\ox \lambda_{k\ell}T\ox \vth_{d,c} \right) \\
&- \sum_{ij} \left( 1\ox a_{ij}\ox T\lambda_{ij}T\ox \vth_{d,d} \right) + \sum_{k\ell} \left( a_{k\ell}\ox 1\ox \lambda_{k\ell}T\ox \vth_{d,d} \right) .
\end{align*}

The next step is to rewrite the rank one operators $\vth_{\alpha,\beta}$ in the right way.
Let us investigate what the action of the antipode $S$ looks like in terms of the self-adjoint generators from Section \ref{ssec:Gen}.
A quick computation yields
\begin{align*}
S(a^u_{ij}) &= a^u_{ji} ,& S(b^u_{ij}) &= - b^u_{ji} ,& S(c^u_{ij}) &= - c^u_{ji} ,& S(d^u_{ij}) = - d^u_{ji} .
\end{align*}
Compare this with
\begin{align*}
(E_{ij}^a)^* &= E_{ji}^a ,& (E_{ij}^b)^* &= - E_{ji}^b ,& (E_{ij}^c)^* &= - E_{ji}^c ,& (E_{ij}^d)^* = - E_{ji}^d .
\end{align*}
Thus write $\Gamma$ for the linear extension of the map $\Gamma\tau_a = \tau_a$, $\Gamma\tau_{b,c,d} = -\tau_{b,c,d}$ on $M_2(\C)$.
Recall the operator $U$ from Section \ref{ssec:FOQG}, which was induced by the unitary antipode $R$.
As we are in the unimodular case, $R$ is the same as $S$.
Hence we can decompose $U = (T\ox \Gamma)$ on $p_1 H\cong M_{2N}(\C)\cong M_2(\C)\ox M_N(\C)$ when we evaluate in $a^u_{ij},\dots,d^u_{ij}$.

We have already written the $M_N(\C)$ leg of $\partial F^{(2)}$ in terms of multiplication operators and transposes, so this suggests that we should find expressions for $\vth_{\alpha,\beta}$ in terms of $\lambda_{a,b,c,d}$ (left multiplication by $\tau_{a,b,c,d}$), $\Gamma$, and $P_{a,b,c,d}$ which are the projections onto $\tau_{a,b,c,d}$ in $M_2(\C)$.
For example, $\vth_{d,b} = \Gamma\lambda_c\Gamma P_b = -\lambda_c\Gamma P_b$.

With this the above relations become
\begin{align*}
\partial\left( F^{(2)}_a\tau_a \right) = &+ \sum_{ij} \left( 1\ox a_{ij}\ox T\lambda_{ij}T\ox \Gamma\lambda_a \Gamma P_a \right) + \sum_{k\ell} \left( a_{k\ell}\ox 1\ox\lambda_{k\ell}T\ox \lambda_a \Gamma P_a \right) \\
&+ \sum_{ij} \left( 1\ox b_{ij}\ox T\lambda_{ij}T\ox \Gamma\lambda_b \Gamma P_b \right) + \sum_{k\ell} \left( b_{k\ell}\ox 1\ox\lambda_{k\ell}T\ox \lambda_b \Gamma P_b \right) \\
&+ \sum_{ij} \left( 1\ox c_{ij}\ox T\lambda_{ij}T\ox \Gamma\lambda_c \Gamma P_c \right) + \sum_{k\ell} \left( c_{k\ell}\ox 1\ox\lambda_{k\ell}T\ox \lambda_c \Gamma P_c \right) \\
&+ \sum_{ij} \left( 1\ox d_{ij}\ox T\lambda_{ij}T\ox \Gamma\lambda_d \Gamma P_d \right) + \sum_{k\ell} \left( d_{k\ell}\ox 1\ox\lambda_{k\ell}T\ox \lambda_d \Gamma P_d \right) ,
\end{align*}
\begin{align*}
\partial\left( F^{(2)}_b\tau_b \right) = &- \sum_{ij} \left( 1\ox b_{ij}\ox T\lambda_{ij}T\ox \Gamma\lambda_b \Gamma P_a \right) - \sum_{k\ell} \left( b_{k\ell}\ox 1\ox \lambda_{k\ell}T\ox \lambda_b \Gamma P_a \right) \\
&- \sum_{ij} \left( 1\ox a_{ij}\ox T\lambda_{ij}T\ox \Gamma\lambda_a \Gamma P_b \right) - \sum_{k\ell} \left( a_{k\ell}\ox 1\ox \lambda_{k\ell}T\ox \lambda_a \Gamma P_b \right) \\
&- \sum_{ij} \left( 1\ox d_{ij}\ox T\lambda_{ij}T\ox \Gamma\lambda_c \Gamma P_d \right) - \sum_{k\ell} \left( d_{k\ell}\ox 1\ox \lambda_{k\ell}T\ox \lambda_d \Gamma P_c \right) \\
&- \sum_{ij} \left( 1\ox c_{ij}\ox T\lambda_{ij}T\ox \Gamma\lambda_d \Gamma P_c \right) - \sum_{k\ell} \left( c_{k\ell}\ox 1\ox \lambda_{k\ell}T\ox \lambda_c \Gamma P_d \right) ,
\end{align*}
\begin{align*}
\partial\left( F^{(2)}_c\tau_c \right) = &- \sum_{ij} \left( 1\ox c_{ij}\ox T\lambda_{ij}T\ox \Gamma\lambda_c \Gamma P_a \right) - \sum_{k\ell} \left( c_{k\ell}\ox 1\ox \lambda_{k\ell}T\ox \lambda_c \Gamma P_a \right) \\
&- \sum_{ij} \left( 1\ox d_{ij}\ox T\lambda_{ij}T\ox \Gamma\lambda_d \Gamma P_b \right) - \sum_{k\ell} \left( d_{k\ell}\ox 1\ox \lambda_{k\ell}T\ox \lambda_d \Gamma P_b \right) \\
&- \sum_{ij} \left( 1\ox a_{ij}\ox T\lambda_{ij}T\ox \Gamma\lambda_a \Gamma P_c \right) - \sum_{k\ell} \left( a_{k\ell}\ox 1\ox \lambda_{k\ell}T\ox \lambda_a \Gamma P_c \right) \\
&- \sum_{ij} \left( 1\ox b_{ij}\ox T\lambda_{ij}T\ox \Gamma\lambda_b \Gamma P_d \right) - \sum_{k\ell} \left( b_{k\ell}\ox 1\ox \lambda_{k\ell}T\ox \lambda_b \Gamma P_d \right) , \\
\end{align*}
\begin{align*}
\partial\left( F^{(2)}_d\tau_d \right) = &- \sum_{ij} \left( 1\ox d_{ij}\ox T\lambda_{ij}T\ox \Gamma\lambda_d \Gamma P_a \right) - \sum_{k\ell} \left( d_{k\ell}\ox 1\ox \lambda_{k\ell}T\ox \lambda_d \Gamma P_a \right) \\
&- \sum_{ij} \left( 1\ox c_{ij}\ox T\lambda_{ij}T\ox \Gamma\lambda_c \Gamma P_b \right) - \sum_{k\ell} \left( c_{k\ell}\ox 1\ox \lambda_{k\ell}T\ox \lambda_c \Gamma P_b \right) \\
&- \sum_{ij} \left( 1\ox b_{ij}\ox T\lambda_{ij}T\ox \Gamma\lambda_b \Gamma P_c \right) - \sum_{k\ell} \left( b_{k\ell}\ox 1\ox \lambda_{k\ell}T\ox \lambda_b \Gamma P_c \right) \\
&- \sum_{ij} \left( 1\ox a_{ij}\ox T\lambda_{ij}T\ox \Gamma\lambda_a \Gamma P_d \right) - \sum_{k\ell} \left( a_{k\ell}\ox 1\ox \lambda_{k\ell}T\ox \lambda_a \Gamma P_d \right) .
\end{align*}
Since
\begin{align*}
\partial F^{(2)} = \partial\left( F^{(2)}_a \tau_a \right) - \partial\left( F^{(2)}_b \tau_b \right) - \partial\left( F^{(2)}_c \tau_c \right) - \partial\left( F^{(2)}_d \tau_d \right) ,
\end{align*}
we obtain the compact formula
\begin{align*}
\partial F^{(2)} = \sum_{ij}^\alpha \left( 1\ox \alpha_{ij}\ox \left[ \left( T\ox \Gamma \right) \lambda_{ij}^\alpha \left( T\ox \Gamma \right) \right] \right) + \sum_{k\ell}^\beta \left( \beta_{k\ell}\ox 1\ox \left[ \lambda_{k\ell}^\beta \left( T\ox \Gamma \right) \right] \right) ,
\end{align*}
where $\lambda_{ij}^\alpha = \lambda_{ij}\otimes\lambda_\alpha$.

By the same techniques it can be shown that
\begin{align*}
\partial F^{(1)} = &+ \sum_{ij} \left( 1\ox a_{ij}\ox \left[ \left( T\ox \Gamma \right) \lambda_{ji}^a \right] \right) + \sum_{k\ell} \left( a_{k\ell}\ox 1\ox \lambda_{\ell k}^a \right) \\
&- \sum_{ij} \left( 1\ox b_{ij}\ox \left[ \left( T\ox \Gamma \right) \lambda_{ji}^b \right] \right) - \sum_{k\ell} \left( b_{k\ell}\ox 1\ox \lambda_{\ell k}^b \right) \\
&- \sum_{ij} \left( 1\ox c_{ij}\ox \left[ \left( T\ox \Gamma \right) \lambda_{ji}^c \right] \right) - \sum_{k\ell} \left( c_{k\ell}\ox 1\ox \lambda_{\ell k}^c \right) \\
&- \sum_{ij} \left( 1\ox d_{ij}\ox \left[ \left( T\ox \Gamma \right) \lambda_{ji}^d \right] \right) - \sum_{k\ell} \left( d_{k\ell}\ox 1\ox \lambda_{\ell k}^d \right) .
\end{align*}

Now we evaluate the `formal' expressions above in the `actual' operators.
Let us start with $\partial F^{(1)}$.
Note that we are taking $a^u_{ij},\dots,d^u_{ij}$ to act on $H$, i.e.\ as elements of $C^*_r\bb{F}O^J_{2N} \subset \s{L}\bb{F}O^J_{2N}$.
Due to the bimodule structure on $\s{C}$, elements in the first tensor leg act from the left, but in the second leg they act from the right.
It is simple to check that in the unimodular case, the right multiplication $\rho$ on $H$ of $x\in C^*_r\bb{F}O^J_{2N}$ can be written $\rho(x) = US(x)U$.

Keeping in mind the identification of $U$ restricted to $p_1 H$ with $(T\ox\Gamma)$ discussed above,
\begin{align*}
\partial F^{(1)}(A^u,\dots,D^u) = &+ \sum_{ij} \left( 1\ox U\ox U \right) \left( 1\ox a^u_{ji}\ox \lambda_{ji}^a \right) \left( 1\ox U \ox 1 \right) + \sum_{k\ell} \left( a^u_{k\ell}\ox 1\ox \lambda_{\ell k}^a \right) \\
&+ \sum_{ij} \left( 1\ox U\ox U \right) \left( 1\ox b^u_{ji}\ox \lambda_{ji}^b \right) \left( 1\ox U \ox 1 \right) - \sum_{k\ell} \left( b^u_{k\ell}\ox 1\ox \lambda_{\ell k}^b \right) \\
&+ \sum_{ij} \left( 1\ox U\ox U \right) \left( 1\ox c^u_{ji}\ox \lambda_{ji}^c \right) \left( 1\ox U \ox 1 \right) - \sum_{k\ell} \left( c^u_{k\ell}\ox 1\ox \lambda_{\ell k}^c \right) \\
&+ \sum_{ij} \left( 1\ox U\ox U \right) \left( 1\ox d^u_{ji}\ox \lambda_{ji}^d \right) \left( 1\ox U \ox 1 \right) - \sum_{k\ell} \left( d^u_{k\ell}\ox 1\ox \lambda_{\ell k}^d \right) ,
\end{align*}
as an element of $B(H\ox H\ox p_1 H)$.
This can be written more compactly as
\begin{align*}
\partial F^{(1)}(A^u,\dots,D^u) = &+ \left( 1\ox U\ox U \right) \left[ \sum_{ij}^\alpha 1\ox \alpha^u_{ij}\ox \lambda_{ij}^\alpha \right] \left( 1\ox U\ox 1 \right) \\
&+ \sum_{k\ell} \left[ a^u_{k\ell}\ox 1\ox \lambda_{\ell k}^a - b^u_{k\ell}\ox 1\ox \lambda_{\ell k}^b - c^u_{k\ell}\ox 1\ox \lambda_{\ell k}^c - d^u_{k\ell}\ox 1\ox \lambda_{\ell k}^d \right] .
\end{align*}

Notice that due to Equation \eqref{eq:nefu}, left multiplication by $u$ on $(p_1 H)\ox H$ looks like $\sum_{ij}^\alpha(\lambda_{ij}^\alpha \ox\alpha^u_{ij})$.
This is also the restriction of the multiplicative unitary $V$ to $(p_1 H)\ox H$ by the decomposition discussed in Section \ref{ssec:CoRep}.
Thus, using leg numbering notation and recalling also Equation \eqref{eq:nefus} yields
\begin{align*}
\partial F^{(1)}(A^u,\dots,D^u) = \left( 1\ox U\ox U \right) V_{32} \left( 1\ox U\ox 1 \right) + V_{31}^* .
\end{align*}
Similarly
\begin{align*}
\partial F^{(2)}(A^u,\dots,D^u) = \left( 1\ox U\ox U \right) V_{32}^* \left( 1\ox U\ox U \right) + V_{31} \left( 1\ox 1\ox U \right) .
\end{align*}

Setting $W = V_{31}(1\ox U\ox U) V_{32}(1\ox U\ox 1)$, it is now a simple matter to see that
\begin{align*}
\partial F^{(1)}(A^u,B^u,C^u,D^u)^* \partial F^{(1)}(A^u,B^u,C^u,D^u) = 2 + 2\fr{Re}\left[W\right].
\end{align*}
For $F^{(2)}(A^u,\dots,D^u)$ it holds that
\begin{align*}
\partial F^{(2)}(A^u,\dots,D^u)^* \partial F^{(2)}(A^u,\dots,D^u) = 2 + 2 \fr{Re}\left[ (1\ox U\ox U)V_{32} (1\ox U\ox U) V_{31} (1\ox 1\ox U)\right] ,
\end{align*}
which reduces to the desired result upon commuting $V_{31}$ with the terms in front of it.
This is allowed because the two terms only act simultaneously on the third tensor leg, where the terms lie in $Uc_0(\bb{F}O^J_{2N})U$ and $c_0(\bb{F}O^J_{2N})$ respectively, which commute.
One way to check this is to use the fact that $c_0(\bb{F}O^J_{2N})$ can be recovered from $V$ by applying the slice maps $(\iota\ox\vph)(V)$, with $\vph$ coming from the predual of $B(H)$, and taking the closed linear span.
\end{proof}

\subsection{$1$-Boundedness}\label{ssec:1B}

In this section we prove $1$-boundedness of the generator set $a^u_{ij},\dots,d^u_{ij}$.
Given the calculation of $\partial F(A^u,\dots,D^u)$ from the previous section, the rest of the arguments are the same as those for the case $\bb{F}O_M$ covered in \cite{BV18}, but we reproduce some of them here for convenience and completeness.

It remains to determine the rank of $\partial F(A^u,\dots,D^u)$ and to show that it is of determinant class.

\begin{lem}\label{lem:DFR}
$\mathrm{rank} ~ \partial F(A^u,\dots,D^u) = (2N)^2-1$
\end{lem}

\begin{proof}
The proof of Lemma 4.1 of \cite{BV18}, where the rank of this operator for $\bb{F}O_M$ is computed, goes through unchanged, as the $L^2$-Betti numbers of $\bb{F}O^J_{2N}$ were shown to also vanish in \cite[Theorem 6.6]{B13} (but see also \cite[Section 5]{Ve12}).
\end{proof}

\begin{thm}[cf.\ {\cite[Theorem 3.5]{BV18}}]\label{thm:EROIDC}
Let $\Theta = U_1 V_{21}U_1 U_2$ be the edge-reversing operator on the quantum Cayley tree of $\bb{F}O^J_{2N}$.
View $1 + \fr{Re}\left[\Theta\right]$ as an operator in $U\s{L}\bb{F}O^J_{2N}U\ox B(p_1 H)$.
Then it is of determinant class with respect to $h\ox\Tr$.
\end{thm}

\begin{proof}
The proof is the same as the one of Theorem 3.5 in \cite{BV18}.
Although it is stated there only for $\bb{F}O_M$, it is also valid for $\bb{F}O^J_{2N}$.
This is due to the fact that the result only depends on the general theory of quantum Cayley graphs \cite{Ve05,Ve12} valid for all $\bb{F}O(Q)$ with $Q\in\mathrm{GL}_M(\C)$, $M\geq 2$, $Q\ol{Q}\in\C I_M$, and qdim$(u) > 2$ (see the remark at the start of Section 3 in \cite{BV18}), and on the Haar state being a trace.
\end{proof}

\begin{prop}\label{prop:DFDC}
$\partial F(A^u,\dots,D^u)^*\partial F(A^u,\cdots,D^u)$ is of determinant class with respect to $h\ox h\ox \Tr$.
\end{prop}

\begin{proof}
Write $\tilde{V} = \Sigma(1\ox U)V(1\ox U)\Sigma$ and notice that $W = V_{31}U_2 U_3 V_{32}U_2$.
We will conjugate $W$ by unitaries $\Omega$ as $\Omega^* W\Omega$ to relate it to $\Theta$.
First conjugate by $U_2\Sigma_{23}$ to obtain
\begin{align*}
\Sigma_{23} U_2 V_{31} U_2 U_3 V_{32} U_2 U_2 \Sigma_{23} = U_3 V_{21} U_3 U_2 \Sigma_{23} V_{32} \Sigma_{23} = V_{21} U_2 V_{23} .
\end{align*}
Next, conjugate by $U_1$ to find
\begin{align*}
U_1 V_{21} U_2 V_{23} U_1 = U_1 V_{21} U_1 U_2 V_{23} = \Sigma_{12} U_2 V_{12} U_2 \Sigma_{12} U_2 V_{23} = \tilde{V}_{12} U_2 V_{23} .
\end{align*}
Finally, conjugate by $V_{23}^*V_{13}^*$ to arrive at
\begin{align*}
V_{13} V_{23} \tilde{V}_{12} U_2 V_{23} V_{23}^* V_{13}^* = V_{13}V_{23}\tilde{V}_{12}U_2V_{13}^* .
\end{align*}
Now use the formula $V_{13}V_{23}\tilde{V}_{12} = \tilde{V}_{12}V_{13}$ of Baaj and Skandalis, which can be found in Proposition 6.1 of \cite{BS93}.
Thus
\begin{align*}
V_{13}V_{23}\tilde{V}_{12}U_2V_{13}^* = \tilde{V}_{12}V_{13} U_2 V_{13}^* = \tilde{V}_{12} U_2 .
\end{align*}
Comparing with the definition of $\Theta$, we see that $\tilde{V}_{12}U_2 = \Theta\ox 1$, and we can conclude that $W$ is unitarily conjugate to $\Theta\ox 1$.
On account of Lemma \ref{lem:DFCTERO}, we also have that $\partial F(A^u,\dots,D^u)^*\partial F(A^u,\cdots,D^u)$ is unitarily conjugate to $4(1+\fr{Re}[\Theta\ox 1])$.

We now consider what happens to $h\ox h\ox \Tr$ under this conjugation process.
The Haar state $h$ is implemented as a vector state by $\xi_0\in H$, and $\Tr$ is implement by some finite sum of vector states by finite dimensionality.
Thus, let $\zeta\in p_1 H$ and compute
\begin{align*}
V_{23}^* V_{13}^* U_1 U_2 \Sigma_{23} (\xi_0\ox\xi_0\ox\zeta) = V_{23}^* V_{13}^* (\xi_0\ox\zeta\ox\xi_0) = V_{23}^* (\xi_0\ox\zeta\ox\xi_0) .
\end{align*}
Hence $h\ox h\ox\Tr$ is transformed into $(h\ox\Tr\ox h)(V_{23}\cdot V_{23}^*)$.
Note that the last two legs of $1+\fr{Re}[\Theta\ox 1]$ lie in the finite dimensional algebra $B(p_1 H)\ox 1$.
By finite dimensionality, $(\Tr\ox h)(V\cdot V^*)$ is dominated by some multiple of the standard trace $(\Tr\ox h)$ on this algebra.
Thus we can use Theorem \ref{thm:EROIDC} to conclude that $1+\fr{Re}[\Theta\ox 1]$ is of determinant class with respect to $(h\ox\Tr\ox h)(V_{23}\cdot V_{23}^*)$.
Therefore $\partial F(A^u,\dots,D^u)^*\partial F(A^u,\cdots,D^u)$ is of determinant class with respect to $h\ox h\ox \Tr$, as desired.
\end{proof}

\begin{cor}\label{cor:1BDNS}
The set of self-adjoint generators $a^u_{ij},\dots,d^u_{ij}$ of $\s{L}\bb{F}O^J_{2N}$ is $1$-bounded.
\end{cor}

\begin{proof}
Combine Lemma \ref{lem:DFR} and Proposition \ref{prop:DFDC} with Theorem \ref{thm:CF1BD}.
\end{proof}

\section{Adding Elements to an $\alpha$-Bounded Set}\label{sec:ADG}

Let $\s{M}$ be a finite von Neumann algebra with faithful normal tracial state $\tau$, and let $X_1,\dots,X_n\in\s{M}$ be self-adjoint.
In this section we prove a lemma that allows us to add certain redundant elements to the set $X_1,\dots,X_n$ while preserving $\alpha$-boundedness.
We achieve this using ideas from Proposition 6.9 in \cite{V2} and its analogue Proposition 6.12 in \cite{V3}.

Let $Y_1,\dots,Y_m$ also be self-adjoint elements in $\s{M}$ such that $Y_1,\dots,Y_m\in W^*(X_1,\dots,X_n)$.
Before stating the lemma, we introduce a distance function that measures how far away the $Y_j$ lie from the von Neumann algebras generated by semicircular perturbations of the $X_i$.
Let $S_1,\dots,S_n$ be a free standard semicircular family, free from the $X_i$, and set
\begin{align*}
d_2(Y_j;X_1,\dots,X_n)(\vep) = \inf\left\{ \norm{Y_j - T}_2 \left| T\in W^*(X_1 + \vep S_1,\dots,X_n + \vep S_n) \right. \right\} .
\end{align*}

\begin{lem}\label{prop:AADG}
Let $\s{M}$ be a finite von Neumann algebra with faithful normal tracial state $\tau$.
Suppose that $X_1,\dots,X_n$ and $Y_1,\dots,Y_m$ are self-adjoint elements such that $Y_1,\dots,Y_m\in W^*(X_1,\dots,X_n)$ \emph{(redundancy)}.
Assume moreover that $\vep^{-1}d_2(Y_j;X_1,\dots,X_n)(\vep)$ is bounded around $\vep = 0$ for all $1\leq j\leq m$ \emph{(regularity)}.
Then if $\{X_1,\dots,X_n\}$ is an $\alpha$-bounded set, so is $\{X_1,\dots,X_n,Y_1,\dots,Y_m\}$.
\end{lem}

\begin{proof}
Note that it suffices to prove the case $m=1$.
Without loss of generality we can extend $S_1,\dots,S_n$ to a free standard semicircular family $S_1,\dots,S_{n+1}$, still free from the $X_i$.
Recalling Definition \ref{dfn:aBS1B}, we need to show that
\begin{align*}
\limsup_{\vep\rightarrow 0} \left[ \chi\left( X_1 + \vep S_1,\dots,X_n + \vep S_n , Y_1 + \vep S_{n+1} : S_1,\dots, S_{n+1} \right) + (n+1-\alpha)\abs{\log\vep} \right] < \infty .
\end{align*}

Write $T_1$ for the conditional expectation of $Y_1$ onto $W^*(X_1 + \vep S_1,\dots,X_n + \vep S_n)$, then by Proposition 1.11 in \cite{V3} and the redundancy assumption we have
\begin{align*}
\chi ( X_1 &+ \vep S_1,\dots,X_n + \vep S_n , Y_1 + \vep S_{n+1} : S_1,\dots, S_{n+1} ) \\
&= \chi(X_1 + \vep S_1,\dots,X_n + \vep S_n, Y_1 - T_1 + \vep S_{n+1} : S_1,\dots,S_{n+1}) .
\end{align*}
By subadditivity ((ii) of Proposition \ref{prop:RMFEP}), we can split this in half as
\begin{align*}
\chi ( X_1 &+ \vep S_1,\dots,X_n + \vep S_n , Y_1 + \vep S_{n+1} : S_1,\dots, S_{n+1} ) \\
&\leq \chi(X_1 + \vep S_1,\dots,X_n + \vep S_n : Y_1 - T_1 + \vep S_{n+1}, S_1, \dots,S_{n+1}) \\
&~~~+ \chi(Y_1 - T_1 + \vep S_{n+1} : X_1 + \vep S_1 , \dots, X_n + \vep S_n,S_1,\dots,S_{n+1}) .
\end{align*}
Consider the first term on the right hand side.
By (iv) of Proposition \ref{prop:RMFEP},
\begin{align*}
\chi(X_1 &+ \vep S_1,\dots,X_n + \vep S_n : Y_1 - T_1 + \vep S_{n+1},S_1,\dots,S_{n+1}) \\
&= \chi(X_1 + \vep S_1,\dots,X_n + \vep S_n : S_1,\dots,S_{n+1}) ,
\end{align*}
as $Y_1-T_1+\vep S_{n+1} \in W^*(X_1 + \vep S_1,\dots, X_n + \vep S_n,S_1,\dots,S_{n+1})$.
To get rid of the trailing semicircular $S_{n+1}$, note that we may apply (iii) of Proposition \ref{prop:RMFEP}, as $S_1,\dots,S_n\in W^*(S_1,\dots,S_{n+1})$.
So
\begin{align*}
\chi ( X_1 &+ \vep S_1,\dots,X_n + \vep S_n , Y_1 + \vep S_{n+1} : S_1,\dots, S_{n+1} ) \\
&\leq \chi(X_1 + \vep S_1,\dots,X_n + \vep S_n : S_1,\dots,S_n) \\
&~~~+ \chi(Y_1 - T_1 + \vep S_{n+1} : X_1 + \vep S_1 , \dots, X_n + \vep S_n,S_1,\dots,S_{n+1}) .
\end{align*}

Let us now focus on the second term on the right hand side.
By (i) of Proposition \ref{prop:RMFEP}, we may replace the relative microstates free entropy by the ordinary microstates free entropy, as we are only after upper bounds.
So
\begin{align*}
\chi(Y_1 - T_1 + \vep S_{n+1} : X_1 + \vep S_1 , \dots, X_n + \vep S_n,S_1,\dots,S_{n+1}) \leq \chi(Y_1 - T_1 + \vep S_{n+1}) .
\end{align*}
Apply the linear change of variable formula for $\chi$ to it (Proposition 3.6 (b) in \cite{V2}), with transformation `matrix' $\vep$.
This yields
\begin{align*}
\chi(Y_1 - T_1 + \vep S_{n+1}) = \log\vep + \chi\left( \vep^{-1}(Y_1 - T_1) + S_{n+1} \right) .
\end{align*}
Using again (i) of Proposition \ref{prop:RMFEP}, we estimate
\begin{align*}
\chi \bigl( \vep^{-1}(Y_1 - T_1) + S_{n+1}& \bigr) \leq \frac{1}{2}\log \left\{ 2\pi e \, \tau\left[ \left( \vep^{-1}(Y_1 - T_1) + S_{n+1} \right)^2 \right] \right\} .
\end{align*}
Thus, if we can control $\norm{\vep^{-1}(Y_1 - T_1) + S_{n+1}}_2$ uniformly in $\vep$, we obtain a constant upper bound.
For this use the triangle inequality and our regularity assumption to obtain
\begin{align*}
\norm{\vep^{-1}(Y_1 - T_1) + S_{n+1}}_2 \leq \vep^{-1} d_2(Y_1;X_1,\dots,X_n)(\vep) + \norm{S_{n+1}}_2 \leq C^\pr .
\end{align*}
In total we have
\begin{align*}
\chi(X_1 &+ \vep S_1,\dots,X_n + \vep S_n, Y_1 + \vep S_{n+1}  : S_1,\dots,S_{n+1}) \\
&\leq \chi(X_1 + \vep S_1,\dots,X_n + \vep S_n : S_1,\dots,S_n) + \log\vep + C .
\end{align*}

To complete the proof, combine all of the above to get
\begin{align*}
\limsup_{\vep\rightarrow 0} \biggl[ \chi &\bigl( X_1 + \vep S_1,\dots,X_n + \vep S_n , Y_1 + \vep S_{n+1} : S_1,\dots,S_{n+1} \bigr) + (n+1-\alpha)\abs{\log\vep} \biggr] \\
&\leq \limsup_{\vep\rightarrow 0} \biggl[ \chi(X_1 + \vep S_1,\dots,X_n + \vep S_n : S_1,\dots,S_n) + \log\vep + C + (n+1-\alpha)\abs{\log\vep} \biggr] \\
&= C + \limsup_{\vep\rightarrow 0} \biggl[ \chi(X_1 + \vep S_1,\dots,X_n + \vep S_n : S_1,\dots,S_n) + (n-\alpha)\abs{\log\vep} + \left( \log\vep + \abs{\log\vep} \right) \biggr] \\
&= C + \limsup_{\vep\rightarrow 0} \biggl[ \chi(X_1 + \vep S_1,\dots,X_n + \vep S_n : S_1,\dots,S_n) + (n-\alpha)\abs{\log\vep} \biggr] \\
&< \infty ,
\end{align*}
as we assumed that $\{X_1,\dots,X_n\}$ is $\alpha$-bounded.
\end{proof}

\begin{rmk}\label{rmk:ARGFDS}
The ideas used in the proof above can be used show that the result is also true when $\delta_0$ is replaced by $\delta^*$.
In fact the proof is simpler.
\end{rmk}

\section{Main Result}\label{sec:MR}

In this section we present our main results and discuss some corollaries.

\begin{thm}\label{thm:MR}
The free orthogonal quantum group von Neumann algebras $\s{L}\bb{F}O^J_{2N}$ are strongly $1$-bounded when $N\geq 2$.
\end{thm}

\begin{proof}
We check that the fundamental character $\chi^u = (\Tr\ox\iota)(u) = 2(a^u_{11}+\dots+a^u_{NN})$ satisfies the requirements of Proposition \ref{prop:AADG}.
The redundancy assumption is trivial, and for the regularity assumption simply note that plugging in the obvious candidate gives a bound
\begin{align*}
d_2(\chi^u;a^u_{11},\dots,d^u_{NN})(\vep) &\leq \norm{ \chi^u - 2\left( a^u_{11} + \vep S^{a}_{11} + \dots + a^u_{NN} + \vep S^{a}_{NN} \right) }_2  \\
&= \norm{ 2\vep S_{11}^{a} + \dots 2\vep S_{NN}^{a} }_2 \\
&\leq 2N\vep .
\end{align*}
Here $S^\alpha_{ij}$ is a free standard semicircular family, free from $a^u_{11},\dots,d^u_{NN}$.
Thus, the set of generators $\{a^u_{11},\dots,d^u_{NN},\chi^u\}$ is also $1$-bounded by Corollary \ref{cor:1BDNS} and Proposition \ref{prop:AADG}.

By \cite{B96}, $\chi^u$ is a semicircular element and hence possesses a continuous density with respect to the Lebesgue measure.
Lemma \ref{lem:fomfe} then allows us to conclude that $\chi(\chi^u)$, i.e.\ the microstates free entropy of the fundamental character, is finite.
We conclude that $\s{L}\bb{F}O^J_{2N}$ is strongly $1$-bounded.
\end{proof}

\begin{rmk}
The proof of Theorem \ref{thm:MR} also extends to strong $1$-boundedness with respect to $\delta^*$ when combined with Remark \ref{rmk:ARGFDS} and recalling that the proof of Corollary \ref{cor:1BDNS} also goes through for $\delta^*$ due to the statement of Theorem \ref{thm:CF1BD}.
\end{rmk}

\begin{cor}\label{cor:CTMR}
Let $Q\in \mathrm{GL}_M(\C)$, $M\geq 3$, be such that $Q\ol{Q} \in \C I_M$ and $\bb{F}O(Q)$ is unimodular, then $\s{L}\bb{F}O(Q)$ is not isomorphic to any finite von Neumann algebra admitting a tuple of self-adjoint generators whose (modified) free entropy dimension exceeds 1.
In particular this excludes being isomorphic to a(n interpolated) free group factor.
\end{cor}

\begin{proof}
By the discussion at the start of section 9.1 in \cite{B17}, it follows that (up to isomorphism) the only two family of matrices satisfying the assumptions are the identity matrices $I_M$, and when $M = 2N$ the standard symplectic matrices $J_{2N}$.
These two cases are covered by Corollary 4.4 in \cite{BV18} and Theorem \ref{thm:MR} above.
\end{proof}

In fact, the class of von Neumann algebras to which $\s{L}\bb{F}O(Q)$ cannot be isomorphic contains all countable free products of finitely generated, diffuse, tracial, Connes embeddable von Neumann algebras by Lemma 3.7 of \cite{J07}.
The free perturbation algebras of Brown \cite{Br05} are also in this class.


\begin{thebibliography}{99}

\bibitem{BS93}
S.\ Baaj \& G.\ Skandalis,
\emph{``Unitaires Multiplicatifs et Dualit\'{e} pour les Produits Crois\'{e}s de C$^*$-alg\`{e}bres"},
Ann.\ Sci.\ \'{E}c.\ Norm.\ Sup\'{e}r.\ (4) 26 (4) (1993) 425--488.

\bibitem{B96}
T.\ Banica,
\emph{``Th{\'e}orie des Repr{\'e}sentations du Group Quantique Compact Libre O$(n)$"},
C.\ R.\ Acad.\ Sci.\ Paris S{\'e}r.\ I Math. 322 (3) (1996) 241--244.

\bibitem{BC07}
T.\ Banica \& B.\ Collins,
\emph{``Integration over Compact Quantum Groups"},
Publ.\ Res.\ Inst.\ Math.\ Sci.\ 43 (2) (2007) 277--302.

\bibitem{BCZJ9}
T.\ Banica, B.\ Collins, \& P.\ Zinn-Justin,
\emph{``Spectral Analysis of the Free Orthogonal Matrix"},
Int.\ Math.\ Res.\ Not.\ IMRN (17) (2009) 3286--3309.

\bibitem{BCG03}
P.\ Biane, M.\ Capitaine, \& A.\ Guionnet,
\emph{``Large Deviation Bounds for Matrix Brownian Motion"},
Invent.\ Math.\ 152 (2) (2003) 433--459.

\bibitem{B13}
J.\ Bichon,
\emph{``Hochschild Homology of Hopf Algebras and free Yetter-Drinfeld Resolutions of the Counit"},
Compos.\ Math.\ 149 (2013), no.\ 4, 658--678.

\bibitem{B12}
M.\ Brannan,
\emph{``Approximation Properties for Free Orthogonal and Free Unitary Quantum Groups"},
J.\ Reine Angew.\ Math.\ 672 (2012) 223--251.

\bibitem{B14}
M.\ Brannan,
\emph{``Strong Asymptotic Freeness for Free Orthogonal Quantum Groups"},
Canad.\ Math.\ Bull.\ 57 (4) (2014) 708--720.

\bibitem{B17}
M.\ Brannan,
\emph{``Approximation Properties for Locally Compact Quantum Groups"},
appearing in \emph{``Topological quantum groups"}, 185-232, Banach Center Publ., 111, Polish Acad.\ Sci.\ Inst.\ Math., Warsaw, 2017.

\bibitem{BCV17}
M.\ Brannan, B.\ Collins, \& R.\ Vergnioux,
\emph{``The Connes Embedding Property for Quantum Group von Neumann Algebras"},
Trans.\ Amer.\ Math.\ Soc.\ 369 (2017), no.\ 6, 3799--3819.

\bibitem{BV18}
M.\ Brannan \& R.\ Vergnioux,
\emph{``Orthogonal Free Quantum Group Factors are Strongly 1-Bounded"},
Adv.\ Math.\ 389 (2018), 133--156.

\bibitem{Br05}
N.\ P.\ Brown,
\emph{``Finite Free Entropy and Free Group Factors"},
Int.\ Math.\ Res.\ Not.\ IMRN no.\ 28 (2005), 1709--1715.

\bibitem{C18}
M.\ Caspers,
\emph{``Gradient Forms and Strong Solidity of Free Quantum Groups"},
Preprint, arXiv:1802.01968v3, 2018.

\bibitem{DCFY14}
K.\ De Commer, A.\ Freslon, \& M.\ Yamashita,
\emph{``CCAP for Universal Discrete Quantum Groups"},
Comm.\ Math.\ Phys.\ 331 (2014), no.\ 2, 677--701.

\bibitem{FV15}
P.\ Fima \& R.\ Vergnioux,
\emph{``A cocycle in the adjoint representation of the orthogonal free quantum groups"},
Int.\ Math.\ Res.\ Not.\ IMRN (2015), no.\ 20, 10069--10094.

\bibitem{F13}
A.\ Freslon,
\emph{``Examples of Weakly Amenable Discrete Quantum Groups"},
J.\ Funct.\ Anal.\ 265 (9) (2013) 2164--2187.

\bibitem{I15}
Y.\ Isono,
\emph{``Examples of Factors Which Have no Cartan Subalgebras"},
Trans.\ Amer.\ Math.\ Soc.\ 367 (11) (2015) 7917--7937.

\bibitem{J07}
K.\ Jung,
\emph{``Strongly 1-Bounded von Neumann Algebras"},
Geom.\ Funct.\ Anal.\ Vol.\ 17 (2007) 1180--1200.

\bibitem{J16}
K.\ Jung,
\emph{``The Rank Theorem and $L^2$-Invariants in Free Entropy: Global Upper Bounds"},
Preprint, arXiv:1602.04726, 2016.

\bibitem{NT13}
S.\ Neshveyev \& L.\ Tuset,
\emph{``Compact Quantum Groups and Their Representation Categories"},
Cours Sp\'{e}cialis\'{e}s, vol.\ 20, Soci\'{e}t\'{e} Math\'{e}matique de France, Paris, 2013.

\bibitem{S16}
D.\ Shlyakhtenko,
\emph{``Von Neumann Algebras of Sofic Groups with $\beta_1^{(2)} = 0$ are strongly 1-bounded"},
Preprint, arXiv:1604.08606, 2016.

\bibitem{VV07}
S.\ Vaes \& R.\ Vergnioux,
\emph{``The Boundary of Universal Discrete Quantum Groups, Exactness, and Factoriality"},
Duke Math.\ J.\ 140 (2007), no.\ 1, 35--84.

\bibitem{VDW96}
A.\ Van Daele \& S.\ Wang,
\emph{``Universal Quantum Groups"},
Internat.\ J.\ Math.\ 7 (2) (1996) 255--263.

\bibitem{Ve05}
R.\ Vergnioux,
\emph{``Orientations of Quantum Cayley Trees and their Applications"},
J.\ Reine Angew.\ Math.\ 580 (2005) 101--138.

\bibitem{Ve12}
R.\ Vergnioux,
\emph{``Paths in Quantum Cayley Trees and $L^2$-Cohomology"},
Adv.\ Math.\ 229 (5) (2012) 2686--2711.

\bibitem{V2}
D.\ Voiculescu,
\emph{``The Analogues of Entropy and of Fisher's Information Measure in Free Probability Theory, II"},
Invent.\ Math.\ 118, 411--440 (1994).

\bibitem{V3}
D.\ Voiculescu,
\emph{``The Analogues of Entropy and of Fisher's Information Measure in Free Probability Theory III: The Absence of Cartan Subalgebras"},
Geometric and Functional Analysis Vol.\ 6, no.\ 1 (1996).

\bibitem{V5}
D.\ Voiculescu,
\emph{``The Analogues of Entropy and of Fisher's Information Measure in Free Probability Theory. V.\ Noncommutative Hilbert Transforms"},
Invent. Math. 132 (1) (1998) 189--227.

\bibitem{W95}
S.\ Wang,
\emph{``Free Products of Compact Quantum Groups"},
Comm.\ Math.\ Phys.\ 167 (3) (1995) 671--692.

\bibitem{Wo87}
S.\ L.\ Woronowicz,
\emph{``Compact Matrix Pseudogroups"},
Comm.\ Math.\ Phys.\ 111 (4) (1987) 613--665.

\bibitem{Wo98}
S.\ L.\ Woronowicz,
``Compact Quantum Groups" in \emph{Sym\'{e}tries Quantiques} (Les Houches, France, 1995),
North-Holland, Amsterdam, 1998, pp.\ 845--884.

\end{thebibliography}
\end{document}